\documentclass[11pt]{article}

\usepackage{amsmath}
\usepackage{amsfonts}
\usepackage{amssymb}
\usepackage[all,2cell]{xy}
\UseTwocells

\newtheorem{proposition}{Proposition}

\newtheorem{theorem}{Theorem}

\newenvironment{proof}{{\bf Proof:}}{$\text{ }\blacksquare$}

\begin{document}

\title{Simplicial approach to derived differential manifolds}
\author{Dennis Borisov, Justin Noel\\ Max-Planck Institute for Mathematics, Bonn, Germany}
\date{\today}
\maketitle

\begin{abstract} 
\noindent Derived differential manifolds are constructed using the usual homotopy theory of simplicial rings of smooth functions. They are proved to be equivalent to derived differential manifolds of finite type, constructed using homotopy sheaves of homotopy rings (D.Spivak), thus preserving the classical cobordism ring. This reduction to the usual algebraic homotopy can potentially lead to virtual fundamental classes beyond obstruction theory.\end{abstract}

\tableofcontents

\section{Introduction}

\underline{The objective} of this paper is to simplify the theory of derived smooth manifolds, developed in \cite{Sp10}, so that one can work with derived manifolds using just simplicial $C^\infty$-rings, instead of homotopy sheaves of homotopy $C^\infty$-rings. 

This complete elimination of sheaf theoretic techniques is possible because the usual softness of structure sheaves on classical manifolds remains true after deriving: every derived differential manifold of finite type\footnote{A derived manifold is of finite type, if it is possible to embed its classical part into $\mathbb R^n$ for some $n\geq0$.} (as defined in \cite{Sp10}) is locally weakly equivalent to a derived manifold, whose structure sheaf is a simplicial diagram of soft sheaves.

\smallskip

In \cite{Sp10} it is proved, that derived differential manifolds, constructed there, have the same cobordism ring as classical manifolds, with intersections being performed by taking homotopy limits. 

Our reduction preserves this property, and thus we provide a model for intersection theory of differential manifolds, using the usual simplicial closed model structure on the category of simplicial $C^\infty$-rings. Intersections are obtained by taking the usual homotopy colimits.

This ability to express intersections homotopically correctly and functorially, by using a single simplicial ring, instead of a homotopy sheaf of such rings, potentially allows one to go beyond obstruction theories in working with virtual fundamental classes. 

\bigskip

\underline{Here is our approach:} we use $C^\infty$-rings as the basis for everything we do in smooth geometry. Such rings are just algebras over a particular algebraic theory (\cite{La63}), we recall the definition at the beginning of Section \ref{FirstSection}. The theory of $C^\infty$-rings is well developed (e.g. \cite{MR91}, \cite{Du81}, \cite{GS03}, \cite{Jo11}, and many others).

Homotopy theory is defined on the category of simplicial $C^\infty$-rings, i.e. simplicial diagrams of $C^\infty$-rings. These rings inherit a simplicial closed model structure from the category of simplicial sets (\cite{Qu67}), and it is this homotopy theory that we use. We recall the definition in Section \ref{SecondSection}.

\smallskip

In \cite{Sp10} homotopy is based on the notion of local weak equivalence. Given a topological space $X$, one considers the category of sheaves of simplicial $C^\infty$-rings, and defines local weak equivalence to be a map that induces an isomorphism between {\it sheaves} of the corresponding homotopy groups. 

This requirement to have an isomorphism on the level of sheaves of homotopy groups and not merely pre-sheaves, is responsible for the local nature of local weak equivalences. In the simplicial setting this notion was defined in \cite{Ja87}. To work with this homotopy theory in terms of closed model categories, one needs hypercovers \cite{DHI}, and they are used in \cite{Sp10}, resulting in the notion of a homotopy sheaf of homotopy $C^\infty$-rings. We recall the definition in Section \ref{ThirdSection}.

\smallskip

Different from algebraic geometry, local weak equivalences in differential geometry can be treated in a very simple manner. Using softness of structure sheaves, we prove that every derived manifold of finite type is locally weakly equivalent to a derived manifold, whose presheaves of homotopy groups are already sheaves. This implies that the functor of global sections maps local weak equivalences to weak equivalences, and so we can go from sheaves to simplicial $C^\infty$-rings without loosing any homotopical structure.

In \cite{Sp10}, instead of pre-sheaves of $C^\infty$-rings, pre-sheaves of {\it homotopy} $C^\infty$-rings are used, i.e. one requires that structure equations are not literally satisfied, but only up to homotopy. One can rectify such homotopy rings into honest $C^\infty$-rings, and it can be done in terms of a Quillen equivalence (\cite{Ba02}, \cite{Be06}). Thus we get the following diagram of adjunctions:
	\begin{equation}\label{ThePlan}\xymatrix{\text{Sheaves of simplicial }C^\infty\text{-rings}\ar@{<->}[r]\ar@{<->}[rd]_{\mathbf \Gamma} & \text{Spivak's theory}\\
	& \text{Simplicial }C^\infty\text{-rings} }\end{equation}
where the categories in the first row come with local weak equivalences, the third category has the usual weak equivalences. Assuming all manifolds are of finite type, each functor maps (local) weak equivalences to (local) weak equivalences. Moreover, all units and counits are (local) weak equivalences.

Going over to simplicial localizations of these categories, we conclude that these adjunctions induce weak equivalences. In \cite{Sp10} derived manifolds are defined by gluing affine derived manifolds, which are just intersections, computed as homotopy limits. Weak equivalences of simplicial localizations, given by (\ref{ThePlan}), show that all of this can be done in the category of simplicial $C^\infty$-rings, using the usual closed model structure.

\bigskip

\underline{Here is the plan} of the paper:  in Section \ref{FirstSection} we describe the correspondence between $C^\infty$-spaces and $C^\infty$-rings. Using softness of the structure sheaves, one proves that this is an equivalence of categories. This material is standard, and we provide it mostly to fix the notation.

In Section \ref{SecondSection} we extend the classical results from Section \ref{FirstSection} to the simplicial case. We show that the standard model category of simplicial $C^\infty$-rings is a model for the category of simplicial $C^\infty$-spaces, where gluing is performed by local weak equivalences.

In Section 4 we recall the construction from \cite{Sp10}, and (assuming finite type) prove that it provides a model for simplicial $C^\infty$-spaces. Thus \cite{Sp10} is reduced, via the category of simplicial $C^\infty$-spaces, to the usual construction, involving homotopy colimits in the category of simplicial $C^\infty$-rings. This correspondence, restricted to good fibrant replacements on the side of \cite{Sp10}, is just the functor of global sections.

\bigskip

We would like to indicate some differences between the approach to derived geometry adopted here, and other places. When one glues by weak equivalences, one can choose the underlying topological spaces to be spectra of the $0$-th homotopy group or the $0$-th component of the structure sheaf. We (and \cite{Sp10}) use the former, while \cite{CK01},\cite{CK02},\cite{TV05},\cite{TV08} use the latter.

In \cite{Jo11b} another approach to derived manifolds is used, where instead of sheaves of simplicial $C^\infty$-rings, one uses sheaves of presentations of $C^\infty$-rings. We believe that constructions of \cite{Jo11} can be obtained from ours by truncating simplicial sets from level $2$ and up. In particular, this would allow one to reformulate the theory of \cite{Jo11b} in terms of $C^\infty$-rings, and their modules. We will pursue this elsewhere.

\bigskip

{\bf Acknowledgements:} This work appeared as a result of Derived Differential Geometry Seminar at Max Planck Institute for Mathematics, in the Summer-Fall, 2011. We would like to thank all the participants for many fruitful discussions in and outside the seminar. The first author is especially grateful to Barbara Fantechi and Timo Sch\"{u}rg for explaining the correspondence between obstruction theory and transversality of derived intersections. 


\section{$C^\infty$-rings and $C^\infty$-spaces}\label{FirstSection}

Let $\mathcal C^\infty\mathcal R$ be the category of product preserving functors 
	\begin{equation}\xymatrix{A:\mathcal C^\infty\ar[r] & Set,}\end{equation} 
where $\mathcal C^\infty$ has $\{\mathbb R^n\}_{n\geq 0}$ as objects, and smooth maps as morphisms. Clearly, any such $A$ is determined (up to a unique isomorphism) by the set $A(\mathbb R)$ and the action of $\{C^\infty(\mathbb R^n)\}_{n\geq 0}$ on $A(\mathbb R)$, making it into a \underline{$C^\infty$-ring}. We will write $A$ to mean both the functor and the corresponding $C^\infty$-ring.

As an example consider a smooth manifold $X$, i.e. a Hausdorff, second countable space with smooth Euclidean atlas of bounded dimension. The set of smooth functions on $X$ is a $C^\infty$-ring, moreover the assignment $X\mapsto C^\infty(X)$ is a full and faithful functor (e.g. \cite{MR91}).

\smallskip

The forgetful functor $\mathcal C^\infty\mathcal R\rightarrow Set$, defined by $A\mapsto A(\mathbb R)$, has a left adjoint 
	\begin{equation}\xymatrix{S\ar@{|->}[r] & C^\infty(\mathbb R^S),}\end{equation}
where $C^\infty(\mathbb R^S)$ is the ring of \underline{smooth finite functions} on $\mathbb R^S:=Hom_{Set}(S,\mathbb R)$, i.e. functions that factor through a projection $\mathbb R^S\rightarrow\mathbb R^F$, with $F\subseteq S$ finite, and a smooth function $\mathbb R^F\rightarrow\mathbb R$.

A $C^\infty$-ring $A$ is \underline{local}, if $A$ has a unique maximal ideal $\mathfrak m\subset A$, and $A/\mathfrak m\cong\mathbb R$. A typical example of a local $C^\infty$-ring is the ring of germs of smooth functions at the origin of $\mathbb R^n$. A $C^\infty$-ring $A$ is \underline{finitely generated} if $A$ is a quotient of $C^\infty(\mathbb R^n)$. Following \cite{MR91}, we will denote the full subcategory of $\mathcal C^\infty\mathcal R$, consisting of finitely generated $C^\infty$-rings, by $\mathcal L$. For any smooth manifold (second countable and of finite dimension), the $C^\infty$-ring of smooth functions is finitely generated (e.g. \cite{MR91}).

\smallskip

\underline{A $C^\infty$-space}\label{ClassicalDefinition} is a pair $(X,\mathcal O_X)$, where\begin{itemize}
\item[1.] $X$ is a Hausdorff topological space,
\item[2.] $\mathcal O_X$ is a soft sheaf of finitely generated $C^\infty$-rings on $X$,
\item[3.] $\forall p\in X$, the stalk $(\mathcal O_X)_p$ is a local $C^\infty$-ring.\end{itemize}
A morphism of $C^\infty$-spaces $(X,\mathcal O_X)\rightarrow(Y,\mathcal O_Y)$ is given by a pair $(\phi,\phi^\sharp)$, where $\phi:X\rightarrow Y$ is a continuous map, and $\phi^\sharp:\mathcal O_Y\rightarrow\phi_*(\mathcal O_X)$ is a morphism of sheaves of $C^\infty$-rings.\footnote{Note, that the corresponding morphisms between stalks are automatically local, since they are morphisms of local $\mathbb R$-algebras, that have $\mathbb R$ as the residue field.} We will denote the category of $C^\infty$-spaces by $\mathbb G$. 

Note that we require the structure sheaf to be a sheaf of finitely generated $C^\infty$-rings, i.e. our $C^\infty$-spaces are \underline{of finite type}. This requirement is equivalent to demanding that a given space is embeddable into some $\mathbb R^n$, or, in the case of manifolds, that the dimension is finite.

\smallskip

There is an obvious functor 
	\begin{equation}\Gamma:\mathbb G\rightarrow\mathcal L^{op},\quad (X,\mathcal O_X)\mapsto\Gamma(X,\mathcal O_X).\end{equation}
This functor has a right adjoint (\cite{Du81}, theorem 8), that we now construct. Let $A\in\mathcal L$, \underline{spectrum of $A$} is the $C^\infty$-space $(Sp(A),\mathcal O_{Sp(A)})$, where 
	\begin{equation}Sp(A):=Hom_{\mathcal L}(A,\mathbb R),\end{equation} 
equipped with \underline{Zariski topology}, and $\mathcal O_{Sp(A)}$ is given by localization. Here is the explicit description: a basic open subset of $Sp(A)$ is 
	\begin{equation}\label{BasicOpen}U_a:=\{p:A\rightarrow\mathbb R\text{ s.t. }p(a)\neq 0\},\quad a\in A.\end{equation}
Clearly $U_{a_1}\cap U_{a_2}=U_{a_1a_2}$, and hence $\{U_a\}_{a\in A}$ is indeed a basis of a topology. Choosing a presentation $A\cong C^\infty(\mathbb R^n)/\mathfrak A$, one can identify
	\begin{equation}Sp(A)\cong\{p:C^\infty(\mathbb R^n)\rightarrow\mathbb R\text{ s.t. }p(\mathfrak A)=0\},\end{equation}
i.e. $Sp(A)$ consists of \underline{the zeroes of $\mathfrak A$} in $\mathbb R^n$. Every open subset of $\mathbb R^n$ has a characteristic function (\cite{MR91}, lemma I.1.4), therefore, Zariski topology on $Sp(A)$ coincides with the topology, induced from $\mathbb R^n$, and every open subset of $Sp(A)$ is of the form (\ref{BasicOpen}). 

For each $a\in A$, let $A\{a^{-1}\}$ be the \underline{smooth localization of $A$} at $a\neq 0$, i.e. it is obtained by universally inverting (in the category of $C^\infty$-rings) every $a'\in A$, s.t. $\forall p\in Sp(A)$ $p(a)\neq 0\Rightarrow p(a')\neq 0$. We define $\mathcal O_{Sp(A)}$ to be the sheaf, associated to $U_a\mapsto A\{a^{-1}\}$. 

There is another description of $\mathcal O_{Sp(A)}$. Let $\mathcal O_{\mathbb R^n}$ be the sheaf of $C^\infty$-functions on $\mathbb R^n$, and let $\mathfrak a\subseteq\mathcal O_{\mathbb R^n}$ be subsheaf of ideals, defined as follows:
	\begin{equation}f\in\Gamma(U,\mathfrak a)\text{ if and only if }\forall p\in U, f_p\in\mathfrak A_p.\end{equation}
Denoting by $\iota:Sp(A)\subseteq\mathbb R^n$ the inclusion, given by $C^\infty(\mathbb R^n)\rightarrow A$, it is easy to see that 
	\begin{equation}\label{SheafPullback}\mathcal O_{Sp(A)}\cong\iota^*(\mathcal O_{\mathbb R^n})/\iota^*(\mathfrak a),\end{equation} 
and hence $\mathcal O_{Sp(A)}$ is soft, and its stalks are local $C^\infty$-rings. 

Finally, being a subspace of $\mathbb R^n$, $Sp(A)$ is clearly Hausdorff.  Therefore $(Sp(A),\mathcal O_{Sp(A)})\in\mathbb G$, and universal property of localization gives us a functor 
	\begin{equation}\mathbf{Sp}:\mathcal L^{op}\rightarrow\mathbb G,\quad A\mapsto(Sp(A),\mathcal O_{Sp(A)}).\end{equation}
In general, it is not true that $A\cong\Gamma(Sp(A),\mathcal O_{Sp(A)})$. Consider the following example: let $A:=C^\infty(\mathbb R^2)/\mathfrak A$, where $\mathfrak A$ consists of functions that vanish in some product neighborhood of the $y$-axis. Then $\Gamma(Sp(A),\mathcal O_{Sp(A)})=C^\infty(\mathbb R^2)/\widetilde{\mathfrak A}$, where $\widetilde{\mathfrak A}$ is the ideal of functions that vanish in some (arbitrary) neighborhood of the $y$-axis. Clearly $\mathfrak A\neq\widetilde{\mathfrak A}$. 

\smallskip

This example can be generalized. An ideal $\mathfrak A\subseteq C^\infty(\mathbb R^n)$ is called \underline{germ determined},\footnote{In \cite{Du81} such ideals are called {\it ideals of local character}, we adopt the terminology from \cite{MR91}.} if 
	\begin{equation}\forall f\in C^\infty(\mathbb R^n)-\mathfrak A,\exists p\in\mathbb R^n\text{ s.t. }f_p\notin\mathfrak A_p.				\end{equation}
A $C^\infty$-ring $A$ is germ determined,  if $A\cong C^\infty(\mathbb R^n)/\mathfrak A$, with $\mathfrak A$ being germ determined. We will denote the full subcategory of $\mathcal L$, consisting of germ determined $C^\infty$-rings, by $\mathcal G$. The inclusion $\mathcal G\subset\mathcal L$ has a left adjoint, that we will denote by $A\mapsto\widetilde{A}$. Explicitly: $\widetilde{A}\cong C^\infty(\mathbb R^n)/\widetilde{\mathfrak A}$, where $\widetilde{\mathfrak A}$ is the germ determined ideal, generated by $\mathfrak A$. Proof of the following proposition is straightforward (use (\ref{SheafPullback})).

\begin{proposition} \label{SpecGerm} Let $A\in\mathcal L$, then $\forall a\in A$, $\Gamma(U_a,\mathcal O_{Sp(A)})\cong\widetilde{A\{a^{-1}\}}$.\end{proposition}
In particular, if $A$ is germ determined, $\Gamma(Sp(A),\mathcal O_{Sp(A)})\cong A$. Actually, this is part of an equivalence of categories.

\begin{proposition}\label{DiscreteCase} (\cite{Du81}) Let $(X,\mathcal O_X)\in\mathbb G$, then $\Gamma(X,\mathcal O_X)\in\mathcal G$. Moreover, 
	\begin{equation}\mathbf{Sp}:\mathcal G^{op}\rightleftarrows\mathbb G:\Gamma\end{equation} 
is an equivalence of categories.\end{proposition}
\begin{proof} Since objects of $\mathbb G$ are locally ringed spaces, standard technique shows that $\Gamma:\mathbb G\rightarrow\mathcal L^{op}$ is left adjoint to $\mathbf{Sp}:\mathcal L^{op}\rightarrow\mathbb G$. Let $(X,\mathcal O_X)\in\mathbb G$, using this adjunction and Proposition \ref{SpecGerm}, we see that the identity map on $\Gamma(X,\mathcal O_X)$ factors through $\widetilde{\Gamma(X,\mathcal O_X)}$, and therefore $\Gamma(X,\mathcal O_X)\in\mathcal G$.

Let $A:=\Gamma(X,\mathcal O_X)$. To prove that the adjunction $\Gamma:\mathbb G\rightleftarrows\mathcal G^{op}:Sp$  is an equivalence, we need to show that $(\iota,\iota^\sharp):(X,\mathcal O_X)\rightarrow(Sp(A),\mathcal O_{Sp(A)})$, corresponding to the identity on $(X,\mathcal O_X)$, is an isomorphism.  Suppose there was $p:\Gamma(X,\mathcal O_X)\rightarrow\mathbb R$, that did not correspond to a point in $X$. Let $\mathfrak m_p$ be the kernel of $p$. Since $\Gamma(X,\mathcal O_X)$ is finitely generated, all of its maximal ideals, having $\mathbb R$ as the residue field, are finitely generated, and hence there are $\alpha_1,\ldots,\alpha_n\in\mathfrak m_p$ generating $\mathfrak m_p$. By assumption there is no point in $X$, where all $\alpha_i$'s vanish, therefore $\alpha_1^2+\ldots+\alpha_n^2\neq 0$ on $X$, and hence it is invertible, contradiction to existence of $p$. This means that $\iota$ is surjective. Since $X$ is Hausdorff and $\mathcal O_X$ is soft, it is standard to prove, that $\iota$ is also injective and its inverse is continuous. 

Now we identify $X=Sp(A)$. Since $\mathcal O_X$ is soft, $\forall p\in X$, $A=\Gamma(X,\mathcal O_X)\rightarrow(\mathcal O_X)_p$ is surjective, and hence $(\mathcal O_{Sp(A)})_p\rightarrow(\mathcal O_X)_p$ is surjective as well. Let $a\in A$, s.t. $a_p\mapsto 0_p\in(\mathcal O_X)_p$, then there is an open $p\in U\subseteq X$, s.t. $\forall q\in U$ $a_q\mapsto 0_q\in(\mathcal O_X)_p$. Let $b\in A$, s.t. $b_p=1_p\in(\mathcal O_{Sp(A)})_p$, $b_q=0_q\in(\mathcal O_{Sp(A)})_q$ $\forall q\in X-U$. Then, clearly, $(a b)_q\mapsto 0_q\in(\mathcal O_X)_q$ $\forall q\in X$, i.e. $a b\mapsto 0\in\Gamma(X,\mathcal O_X)$. Therefore $a b=0\in A$, and hence $a_p=0\in(\mathcal O_{Sp(A)})_p$, i.e. $\mathcal O_{Sp(A)}\rightarrow\mathcal O_X$ is an isomorphism.\end{proof}

\smallskip

From Proposition \ref{DiscreteCase} it is clear, that for a finitely generated $C^\infty$-ring, being germ determined is equivalent to being {\it geometric}, i.e. being the ring of smooth functions on a $C^\infty$-space. The requirement to be finitely generated is essential, since even the free $C^\infty$-ring $C^\infty(\mathbb R^S)$, with $S$ infinite, is not isomorphic to $\Gamma(Sp(C^\infty(\mathbb R^S)),\mathcal O_{Sp(C^\infty(\mathbb R^S))})$.

Yet, one has the notion of a germ determined $C^\infty$-ring also in the infinitely generated case (\cite{Bo11}). In general, for $A$ to be germ determined is equivalent to $A\rightarrow\Gamma(Sp(A),\mathcal O_{Sp(A)})$ being injective. 

Surjectivity is more complicated. If $Sp(A)$ is paracompact, it is enough for $A$ to be \underline{locally complete},\footnote{In \cite{Jo11} this is called being {\it complete with respect to locally finite sums}.} i.e. if a family $\{a_i\}\subseteq A$ is s.t. $\forall p\in Sp(A)$ only finitely many of $\{(a_i)_q\}$ are not $0$ in a neighborhood of $p$, then $\exists a\in A$, s.t. $\forall p\in Sp(A)$ $a_p=\Sigma(a_i)_p$. Any finitely generated, germ determined $C^\infty$-ring is complete in this sense. This is not always true for the infinitely generated ones.

\smallskip

These notions can be extended to modules over $C^\infty$-rings.\footnote{Modules over $C^\infty$-rings are just modules over the underlying commutative $\mathbb R$-algebras.} An $A$-module $M$ is \underline{germ determined and locally complete} if for any family $\{t_i\}\subseteq M$, s.t $\forall p\in Sp(A)$ almost all of $\{(t_i)_q\}$ are $0$ in a neighborhood of $p$, there is a unique $t\in M$, s.t. $\forall p\text{ }\Sigma(t_i)_p=t_p$.
We will denote the category of such $A$-modules by $\overline{Mod}_A$. Since sheaves of modules over soft sheaves of rings are themselves soft, we have the following conclusion.
\begin{proposition} Let $A$ be a finitely generated $C^\infty$-ring. The functor of global sections defines an equivalence between the category of sheaves of $\mathcal O_{Sp(A)}$-modules and $\overline{Mod}_A$.\end{proposition}


\section{Simplicial $C^\infty$-rings and simplicial $C^\infty$-spaces}\label{SecondSection}

\underline{A simplicial $C^\infty$-ring} is a simplicial diagram 
	\begin{equation}A_\bullet:=(A_k,\{\partial_{k+1,i}:A_{k+1}\rightarrow A_k\}_{i=0}^{k+1},\{\delta_{k,j}:A_k\rightarrow 
	A_{k+1}\}_{j=0}^k)_{k\geq 0}\end{equation}
in the category of $C^\infty$-rings. We don't require any $A_k$ to be finitely generated. Morphisms are the usual natural transformations between simplicial diagrams. We will denote the category of simplicial $C^\infty$-rings by $s\mathcal C^\infty\mathcal R$.

From \cite{Qu67} we know that $s\mathcal C^\infty\mathcal R$ is a simplicial closed model category, where \underline{weak equivalences} and \underline{fibrations} are those morphisms that define respectively weak equivalences and fibrations between the underlying simplicial sets. Since each simplicial $C^\infty$-ring is a simplicial $\mathbb R$-module, it is enough to calculate homotopy groups at one chosen point (e.g. $0$), and, moreover, 
	\begin{equation}\pi_k(A_\bullet,0)\cong H_k(\mathcal M(A_\bullet)),\end{equation} 
where $\mathcal M(A_\bullet)=(\bigoplus A_k,\Sigma(-1)^i\partial_{k+1,i})$ is the \underline{Moore complex} of $A_\bullet$. We will write $\pi_k(A_\bullet)$ instead of $\pi_k(A_\bullet,0)$.

Boundaries $\partial_{1,0},\partial_{1,1}:A_1\rightarrow A_0$ define a $\mathcal C^\infty$-congruence on $A_0$, and $\pi_0(A_\bullet)$ is the quotient of $A_0$ by this congruence. Therefore $\pi_0(A_\bullet)$ is a $C^\infty$-ring. Moreover, since each $A_k$ is in particular a commutative $\mathbb R$-algebra, $\forall k\geq 1$ $\pi_k(A_\bullet)$ is a module over $\pi_0(A_\bullet)$.  

\smallskip

Let $X$ be a topological space, and let $\mathcal O_{\bullet,X}$ be a sheaf of $C^\infty$-rings on $X$. We will denote by $\pi_k(\mathcal O_{\bullet,X})$, $k\geq 0$ the sheaves associated to presheaves  
	\begin{equation}\xymatrix{U\ar@{|->}[r] & \pi_k(\Gamma(U,\mathcal O_{\bullet,X})),\quad k\geq 0.}\end{equation}
Clearly, $\pi_0(\mathcal O_{\bullet,X})$ is a sheaf of $C^\infty$-rings, and $\{\pi_k(\mathcal O_{\bullet,X})\}_{k\geq 1}$ are sheaves of $\pi_0(\mathcal O_{\bullet,X})$-modules.

\underline{A simplicial $C^\infty$-space} is a pair $(X,\mathcal O_{\bullet,X})$, where $X$ is a topological space, and $\mathcal O_{\bullet,X}:=\{\mathcal O_{k,X}\}_{k\geq 0}$ is a sheaf of simplicial $C^\infty$-rings on $X$, s.t. $(X,\pi_0(\mathcal O_{\bullet,X}))$ is a $C^\infty$-space (Section \ref{FirstSection}), i.e. $X$ is Hausdorff, and $\pi_0(\mathcal O_{\bullet,X})$ is a soft sheaf of finitely generated $C^\infty$-rings, whose stalks are local $C^\infty$-rings.

A morphism $(X,\mathcal O_{\bullet,X})\rightarrow(Y,\mathcal O_{\bullet,Y})$ is given by a pair $(\phi,\phi^\sharp_\bullet)$, where $\phi:X\rightarrow Y$ is a continuous map, and $\phi^\sharp_\bullet:=\{\phi^\sharp_k:\mathcal O_{k,Y}\rightarrow\phi_*(\mathcal O_{k,X})\}_{k\geq 0}$ is a morphism of sheaves of simplicial $C^\infty$-rings. We will denote the category of simplicial $C^\infty$-spaces by $s\mathbb G$.

\subsection{Local weak equivalences}

Given $(X,\mathcal O_{\bullet,X})\in s\mathbb G$, we have a $C^\infty$-space $(X,\pi_0(\mathcal O_{\bullet,X}))$ and a sequence $\{\pi_k(\mathcal O_{\bullet,X})\}_{k\geq 1}$ of sheaves of modules on it. We will denote this $C^\infty$-space, together with the sheaves of modules, by $\pi_\bullet(X,\mathcal O_{\bullet,X})$. Every morphism $(\phi,\phi^\sharp_\bullet):(X,\mathcal O_{\bullet,X})\rightarrow(Y,\mathcal O_{\bullet,Y})$ in $s\mathbb G$ induces a morphism 
	\begin{equation}\xymatrix{\pi_\bullet(\phi,\phi^\sharp_\bullet):\pi_\bullet(X,\mathcal O_{\bullet,X})\ar[r] & \pi_\bullet(Y,\mathcal O_{\bullet,Y}).}\end{equation}  
We call $(\phi,\phi^\sharp_\bullet)$ a \underline{local weak equivalence}, if $\pi_\bullet(\phi,\phi^\sharp_\bullet)$ is an isomorphism.

\smallskip

Let $s\mathcal L\subset s\mathcal C^\infty\mathcal R$ be the full subcategory, consisting of  $A_\bullet$, s.t. $\pi_0(A_\bullet)$ is a finitely generated $C^\infty$-ring.\footnote{Note that we do not require any $A_k$ to be finitely generated, only $\pi_0(A_\bullet)$.} There is an obvious functor 
	\begin{equation}\Gamma:s\mathbb G\rightarrow s\mathcal L^{op},\quad (X,\mathcal O_{\bullet,X})\mapsto\Gamma(X,\mathcal O_{\bullet,X}).\end{equation} 
In general $\Gamma$ does not map local weak equivalences in $s\mathbb G$ to weak equivalences in $s\mathcal L^{op}$. It does so for a particular kind of simplicial $C^\infty$-spaces, that we are going to consider next.

\smallskip

\underline{A localized simplicial $C^\infty$-space} is a $C^\infty$-space $(X,\mathcal O_{\bullet,X})$, s.t. stalks of $\mathcal O_{0,X}$ are local $C^\infty$-rings. We will denote by $\overline{s\mathbb G}\subset s\mathbb G$ the full subcategory, consisting of localized simplicial $C^\infty$-spaces. The following proposition explains why $\Gamma$, restricted to $\overline{s\mathbb G}$, maps local weak equivalences to weak equivalences.

\begin{proposition} Let $(X,\mathcal O_{\bullet,X})\in\overline{s\mathbb G}$, then $\mathcal O_{0,X}$ is a soft sheaf.\end{proposition}
\begin{proof} Since $(X,\pi_0(\mathcal O_{\bullet,X}))$ is a $C^\infty$-space, $X$ is paracompact. Therefore (\cite{Go60}, th. II.3.7.2), to prove that $\mathcal O_{0,X}$ is soft, it is enough to show that it is locally soft, i.e. $\forall p\in X$ there is a neighborhood $U\ni p$, s.t. $\forall V_1,V_2\subseteq U$ closed, with $V_1\cap V_2=\emptyset$, there is $\alpha\in\Gamma(U,\mathcal O_{0,X})$, s.t. $\alpha_{V_1}=0$, $\alpha_{V_2}=1$.

Let $p\in X$, and let $\{f_i\}_{1\leq i\leq n}$ be a set of generators of $\Gamma(X,\pi_0(\mathcal O_{\bullet,X}))$ as a $C^\infty$-ring. Since $(\mathcal O_{0,X})_p\rightarrow(\pi_0(\mathcal O_{\bullet,X}))_p$ is surjective, there is a neighborhood $W\ni p$, s.t. $\{f_i|_W\}_{i=1}^n$ are in the image of $\Gamma(W,\mathcal O_{0,X})\rightarrow\Gamma(W,\pi_0(\mathcal O_{\bullet,X}))$. Choose a closed neighborhood $U\ni p$, s.t. $U\subseteq W$, we claim that $U$ satisfies the conditions of theorem II.3.7.2 in \cite{Go60}.

Indeed, since $\pi_0(\mathcal O_{\bullet,X})$ is soft, $\Gamma(X,\pi_0(\mathcal O_{\bullet,X}))\rightarrow\Gamma(U,\pi_0(\mathcal O_{\bullet,X}))$ is surjective, and hence $\Gamma(U,\pi_0(\mathcal O_{\bullet,X}))$ is generated by the images of $\{f_i\}_{i=1}^n$. Therefore $\Gamma(U,\mathcal O_{0,X})\rightarrow\Gamma(U,\pi_0(\mathcal O_{\bullet,X}))$ is surjective. We can choose a finitely generated $C^\infty$-subring $A\subseteq\Gamma(U,\mathcal O_{0,X})$, s.t. $A\rightarrow\Gamma(U,\pi_0(\mathcal O_{\bullet,X}))$ is surjective. We claim that $U\cong Sp(A)$ as topological spaces.

Suppose not, surjectivity of $A\rightarrow\Gamma(U,\pi_0(\mathcal O_{\bullet,X}))$ implies that $U\subseteq Sp(A)$ as a closed subspace, hence $Sp(A)-U$ is non-empty and open. Therefore, $\exists a\in A$, s.t. $a\neq 0$, $a_U=0$, when considered as a section of $\mathcal O_{Sp(A)}$. Since $A$ is finitely generated, $\exists b\in A$, s.t. $a b=0$ and $b(q)=1$ $\forall q\in U$. Since $(X,\mathcal O_{\bullet,X})$ is localized, $b$ is invertible in $\Gamma(U,\mathcal O_{0,X})$, and therefore $a=0$ as a section of $\mathcal O_{0,X}$, contradiction to $A$ being a $C^\infty$-subring. So $U\cong Sp(A)$.

From $U\cong Sp(A)$ and $A$ being finitely generated, we conclude that $\forall V_1,V_2\subseteq U$ closed, s.t. $V_1\cap V_2=\emptyset$, there is $\alpha\in A$, s.t. $\alpha_{V_1}=0$, $\alpha_{V_2}=1$, where we consider $\alpha$ as a section of $\mathcal O_{Sp(A)}$. Since $(X,\mathcal O_{\bullet,X})$ is localized, $\alpha$ has the same properties, when considered as a section of $\mathcal O_{0,X}$. 
\end{proof}

\smallskip

Let $(X,\mathcal O_{\bullet,X})\in\overline{s\mathbb G}$. Since $\mathcal O_{0,X}$ is soft, all sheaves of boundaries in $\mathcal M(\mathcal O_{\bullet,X})$ are soft as well, and hence (see e.g. \cite{Go60}, theorem II.3.5.2) 
	\begin{equation}\label{GammaCommutes}\forall k\geq 0\quad\Gamma(X,\pi_k(\mathcal O_{\bullet,X}))\cong\pi_k(\Gamma(X,\mathcal O_{\bullet,X})).\end{equation}
This immediately implies the following proposition.

\begin{proposition}\label{GlobalSections} The functor of global sections $\Gamma:\overline{s\mathbb G}\rightarrow s\mathcal L^{op}$ maps local weak equivalences to weak equivalences.\end{proposition}
It might appear that localized simplicial $C^\infty$-spaces are very special. However, the following proposition shows that every simplicial $C^\infty$-space can be \underline{localized}, s.t. the result is locally weak equivalent to the original space.

\begin{proposition}\label{LocalizationFunctor} The inclusion $\overline{s\mathbb G}\subset s\mathbb G$ has a right adjoint. Unit of this adjunction is an isomorphism, and counit consists of local weak equivalences.\end{proposition}
\begin{proof} Let $(X,\mathcal O_{\bullet,X})\in s\mathbb G$, and let $p\in X$. The stalk $(\mathcal O_{0,X})_p$ does not have to be local, but $(\pi_0(\mathcal O_{\bullet,X}))_p$ is local. Therefore, since $(\pi_0(\mathcal O_{\bullet,X}))_p$ is a quotient of $(\mathcal O_{0,X})_p$, we get a distinguished 
	\begin{equation}\mathfrak p:(\mathcal O_{0,X})_p\longrightarrow\mathbb R.\end{equation}
Let $(\mathcal O_{0,X})_{\mathfrak p}$ be the localization of $(\mathcal O_{0,X})_p$ at $\mathfrak p$, i.e. it is obtained by universally inverting every $f_p\in(\mathcal O_{0,X})_p$, whose value at $\mathfrak p$ is not $0$. As usual with $C^\infty$-rings, the natural map
	\begin{equation}\label{PointLocalization}\xymatrix{(\mathcal O_{0,X})_p\ar[r] & (\mathcal O_{0,X})_{\mathfrak p}}\end{equation}
is surjective. We will denote kernel of (\ref{PointLocalization}) by $\mathfrak t_p$. For an open $U\subseteq X$, we will call a section $\alpha\in\Gamma(U,\mathcal O_{0,X})$ \underline{trivial}, if $\forall p\in U$ $\alpha_p\in\mathfrak t_p$. It is clear that all trivial sections together comprise a sheaf of ideals $\mathfrak t\subset\mathcal O_{0,X}$. It is easy to check that $\forall p\in X$ $\mathfrak t_p$ is exactly the set of germs of $\mathfrak t$ at $p$. 

We define
	\begin{equation}\overline{\mathcal O}_{\bullet,X}:=\mathcal O_{\bullet,X}/\mathfrak t\mathcal O_{\bullet,X},\end{equation}
where the r.h.s. is taken in the category of sheaves of simplicial $C^\infty$-rings, i.e. we first divide, and then sheafify. Proposition \ref{FullLocalization} gives the middle isomorphism in 
	\begin{equation}\forall p\in X,\quad(\pi_0(\overline{\mathcal O}_{\bullet, X}))_p\cong\pi_0((\overline{\mathcal O}_{\bullet,X})_p)\cong
	\pi_0((\mathcal O_{\bullet,X})_p)\cong(\pi_0(\mathcal O_{\bullet,X}))_p,\end{equation}
and hence $(X,\overline{\mathcal O}_{\bullet,X})\in s\mathbb G$. Since stalks of $\overline{\mathcal O}_{\bullet,X}$ are local $C^\infty$-rings, we have that in fact $(X,\overline{\mathcal O}_{\bullet, X})\in\overline{s\mathbb G}$. It is clear that $(X,\mathcal O_{\bullet, X})\mapsto(X,\overline{\mathcal O}_{\bullet,X})$ extends to a functor $s\mathbb G\rightarrow\overline{s\mathbb G}$, and this functor is right adjoint to $\overline{s\mathbb G}\subset s\mathbb G$.

If $(X,\mathcal O_{\bullet,X})$ is localized, there are no non-zero trivial sections, and hence unit of the adjunction is an isomorphism. Using Proposition \ref{FullLocalization} again, we conclude that the map $\mathcal O_{\bullet,X}\rightarrow\overline{\mathcal O}_{\bullet,X}$ is a weak equivalence stalk-wise, i.e. counit of the adjunction consists of local weak equivalences.\end{proof}

\smallskip

Let $\mathbf\Gamma:s\mathbb G\rightarrow s\mathcal L^{op}$ be composition of the localization functor $s\mathbb G\rightarrow\overline{s\mathbb G}$ and the functor of global sections $\Gamma:\overline{s\mathbb G}\rightarrow s\mathcal L^{op}$. Propositions \ref{GlobalSections} and \ref{LocalizationFunctor} imply that $\mathbf\Gamma$ maps local weak equivalences to weak equivalences. Our next objective is to define a functor, going in the opposite direction.

\subsection{Spectrum of a simplicial $C^\infty$-ring}

Let $A_\bullet\in s\mathcal L$, i.e. $A_\bullet$ is a simplicial $C^\infty$-ring, s.t. $\pi_0(A_\bullet)$ is finitely generated as a $C^\infty$-ring. We define 
	\begin{equation}Sp(A_\bullet):=Sp(\pi_0(A_\bullet)).\end{equation} 
Let $U\subseteq Sp(A_\bullet)$ be open. For any $p\in U$ we have an evaluation
	\begin{equation}\mathfrak p:\xymatrix{A_0\ar[r] & \pi_0(A_\bullet)\ar[r]^p & \mathbb R.}\end{equation}
Let $\mathcal U:=\{a\in A_0\text{ s.t. }\forall p\in U,\mathfrak p(a)\neq 0\}$, we define $A_\bullet\{\mathcal U^{-1}\}$ to be the simplicial $C^\infty$-ring, obtained by universally inverting every $a\in\mathcal U$, and its degeneracies in $A_{\geq 1}$. 
It is clear that $U\mapsto A_\bullet\{\mathcal U^{-1}\}$ is a presheaf of simplicial $C^\infty$-rings on $Sp(A_\bullet)$. Let $\mathcal O_{\bullet,Sp(A_\bullet)}$ be the associated sheaf. 

\begin{proposition} Let  $(Sp(A_\bullet),\mathcal O_{\bullet, Sp(A_\bullet)})$ be defined as above. Then 
	\begin{equation}\pi_0(\mathcal O_{\bullet,Sp(A_\bullet)})\cong\mathcal O_{Sp(\pi_0(A_\bullet))}.\end{equation}\end{proposition}
\begin{proof} Recall that $\mathcal O_{Sp(\pi_0(A_\bullet))}$ is sheafification of the presheaf 
	\begin{equation}U\mapsto\pi_0(A_\bullet)\{[\mathcal U]^{-1}\},\quad U\subseteq Sp(\pi_0(A_\bullet)),\end{equation} 
where $[\mathcal U]\subseteq\pi_0(A_\bullet)$ consists of those elements, that do not vanish at any $p\in U$, i.e. $\mathcal U\subseteq A_0$ is the pre-image of $[\mathcal U]$. Therefore we have a canonical $A_0\{\mathcal U^{-1}\}\rightarrow\pi_0(A_\bullet)\{[\mathcal U]^{-1}\}$, and it obviously factors through $\pi_0(A_\bullet\{\mathcal U^{-1}\})$. 
Using universal properties of localization and sheafification, we arrive at 
	\begin{equation}\label{HomotopyCommutesMorphism}\pi_0(\mathcal O_{\bullet,Sp(A_\bullet)})\rightarrow\mathcal O_{Sp(\pi_0(A_\bullet))}.\end{equation} 
To prove that (\ref{HomotopyCommutesMorphism}) is an isomorphism, we need to look at the stalks. Let $p\in Sp(A_\bullet)$, then $(\mathcal O_{Sp(\pi_0(A_\bullet))})_p\cong\pi_0(A_\bullet)/\mathfrak m_p^g$, $(\mathcal O_{\bullet,Sp(A_\bullet)})_p\cong A_\bullet/\mathfrak m_\mathfrak p^gA_\bullet$,
where $\mathfrak m_p^g\subset\pi_0(A_\bullet)$, $\mathfrak m_\mathfrak p^g\subset A_0$ consist of elements that have $0$ germs at respectively $p$ and $\mathfrak p$ (see Proposition \ref{InvertingModulo}). Let $\mathfrak k$ be the kernel of $A_0\rightarrow\pi_0(A_\bullet)$. It is easy to see that the l.h.s. of (\ref{HomotopyCommutesMorphism}) is $A_0/\mathfrak k+\mathfrak m_\mathfrak p^g$, while the r.h.s. is $(A_0/\mathfrak k)/\mathfrak m_p^g$. It is straightforward to see that $A_0/\mathfrak k+\mathfrak m_\mathfrak p^g=(A_0/\mathfrak k)/\mathfrak m_p^g$. 
\end{proof}

\smallskip

The previous proposition implies that $(Sp(A_\bullet),\mathcal O_{\bullet,A_\bullet})\in s\mathbb G$. Since construction of $\mathcal O_{\bullet,Sp(A_\bullet)}$ involves only inverting elements, it is clearly functorial. Moreover, it is obvious that stalks of $\mathcal O_{0,Sp(A_\bullet)}$ are local $C^\infty$-rings, so we have a functor 
	\begin{equation}\mathbf{Sp}:s\mathcal L^{op}\rightarrow\overline{s\mathbb G},\quad A_\bullet\mapsto(Sp(A_\bullet),\mathcal O_{\bullet,Sp(A_\bullet)}).\end{equation}
In general, it is not true that $A_\bullet$ is weakly equivalent to $\Gamma(Sp(A_\bullet),\mathcal O_{\bullet,Sp(A_\bullet)})$. Even in the simple case when $A_\bullet$ is a discrete $C^\infty$-ring, i.e. a constant simplicial diagram $\Delta^{op}\mapsto A\in\mathcal C^\infty\mathcal R$, if $A$ is not germ determined, then $A\ncong\Gamma(Sp(A),\mathcal O_{Sp(A)})$.

\smallskip

A simplicial $C^\infty$-ring $A_\bullet$ is \underline{geometric}, if it is homotopically finitely generated and germ determined, i.e. if\begin{itemize}
\item[1.] $\pi_0(A_\bullet)$ is a finitely generated, germ determined $C^\infty$-ring,
\item[2.] $\forall k\geq 1$, $\pi_k(A_\bullet)$ is a germ determined, locally complete $\pi_0(A_\bullet)$-module (see Section \ref{FirstSection} for definition of local completeness).\end{itemize}
Let $s\mathcal G\subset s\mathcal L$ be the full subcategory, consisting of geometric simplicial $C^\infty$-rings. If $(X,\mathcal O_{\bullet,X})\in\overline{s\mathbb G}$, from (\ref{GammaCommutes}) we know, that for any $k\geq 0$ $\pi_k(\Gamma(X,\mathcal O_{\bullet,X}))\cong\Gamma(X,\pi_k(\mathcal O_{\bullet, X}))$, and hence $\pi_0(\Gamma(X,\mathcal O_{\bullet,X}))$ is finitely generated, germ determined, and  $\pi_k(\Gamma(X,\mathcal O_{\bullet,X}))$ are all germ determined, locally complete $\pi_0(\Gamma(X,\mathcal O_{\bullet, X}))$-modules. Therefore, the functor of global sections $\Gamma$ maps $\overline{s\mathbb G}$ to $s\mathcal G^{op}$.

\begin{proposition} The functor of global sections $\Gamma:\overline{s\mathbb G}\rightarrow s\mathcal G^{op}$ is left adjoint to $\mathbf{Sp}:s\mathcal G^{op}\rightarrow\overline{s\mathbb G}$. Unit of this adjunction is an isomorphism, the counit consists of weak equivalences.\end{proposition}
\begin{proof} Let $(X,\mathcal O_{\bullet,X})\in\overline{s\mathbb G}$, $A_\bullet\in s\mathcal G$. Any morphism $f_\bullet:A_\bullet\rightarrow\Gamma(X,\mathcal O_{\bullet,X})$ induces $\pi_0(f_\bullet):\pi_0(A_\bullet)\rightarrow\pi_0(\Gamma(X,\mathcal O_{\bullet,X}))\cong\Gamma(X,\pi_0(\mathcal O_{\bullet,X}))$, and hence a continuous map $\phi:X\rightarrow Sp(A_\bullet)$. Since stalks of $\mathcal O_{0,X}$ are local, universal property of localization implies that $f_\bullet$ defines $\phi^\sharp_\bullet:\mathcal O_{\bullet,Sp(A_\bullet)}\rightarrow\phi_*(\mathcal O_{\bullet,X})$. 

Let $\overline{A}_\bullet$ be obtained by universally inverting every $a\in A_0$, s.t. $\mathfrak p(a)\neq 0$ for any $p\in Sp(\pi_0(A_\bullet))$. From Proposition \ref{InvertingModulo} we know that $A_\bullet\rightarrow\overline{A}_\bullet$ is surjective, and hence, using universal property of sheafification and proceeding as in the proof of Proposition \ref{DiscreteCase}, we conclude that $f_\bullet\mapsto(\phi,\phi^\sharp_\bullet)$ is a bijective correspondence. It is clearly functorial, i.e. $\mathbf{Sp}$ is right adjoint to $\Gamma$.

Suppose that $A_\bullet=\Gamma(X,\mathcal O_{\bullet,X})$. Since $\pi_0(A_\bullet)\cong\Gamma(X,\pi_0(\mathcal O_{\bullet,X}))$, obviously $X\cong Sp(A_\bullet)$ as topological spaces. Then, since $\mathcal O_{0,X}$ is soft, and its stalks are local, one proves that $\mathcal O_{\bullet,Sp(A_\bullet)}\cong\mathcal O_{\bullet, X}$ as in the discrete case (Proposition \ref{DiscreteCase}). So unit of the adjunction is indeed an isomorphism.

Now we compare $A_\bullet$ and $\Gamma(Sp(A_\bullet),\mathcal O_{Sp(A_\bullet)})$. From Proposition \ref{FullLocalization} we know that $A_\bullet\rightarrow\overline{A}_\bullet$ is a weak equivalence. It remains to show that sheafification preserves this homotopy type, i.e. that
	\begin{equation}\gamma:\overline{A}_\bullet\rightarrow\Gamma(Sp(A_\bullet),\mathcal O_{Sp(A_\bullet)})\end{equation}
is a weak equivalence. We can assume that $A_\bullet=\overline{A}_\bullet$. Let $\alpha\in\mathcal M(A_\bullet)$ be a cycle, s.t. $\gamma(\alpha)$ is a boundary. This implies that $\forall p\in Sp(A_\bullet)$, $\gamma(\alpha)_p$ is a boundary in $(\mathcal O_{\bullet,Sp(A_\bullet)})_p$. From Proposition \ref{InvertingModulo} we conclude that $(\mathcal O_{\bullet,Sp(A_\bullet)})_p\cong A_\bullet/\mathfrak m_{\mathfrak p}^g A_\bullet$, therefore we can construct $\beta\in A_0$, s.t. $\alpha\beta$ is homologous to $0$, yet $\mathfrak p(\beta)\neq 0$. This implies that $\alpha\mapsto 0$ in $H_\bullet(\mathcal M(A_\bullet))_p$. Since this happens for every $p$, and we assume $\pi_k(A_\bullet)$ to be germ determined $\forall k\geq 0$, $\alpha$ has to be a boundary itself.

Now let $\xi\in\Gamma(Sp(A_\bullet),\mathcal O_{\bullet,Sp(A_\bullet)})$ be a cycle. Let $p\in Sp(A_\bullet)$, then $\xi_p\in(\mathcal O_{\bullet,Sp(A_\bullet)})_p\cong A_\bullet/\mathfrak m_\mathfrak p^gA_\bullet$ is a cycle. It is easy to see that for any $\alpha\in A_0$, whose support is contained in a small open set around $\mathfrak p$, $\alpha\xi_p$ extends to a cycle in $A_\bullet$. Therefore, using partition of unity, we can find a family of cycles $\{\beta_i\}\subseteq A_\bullet$, s.t. $\{\gamma(\beta_i)\}$ is locally finite and
	\begin{equation}\forall p\in Sp(A_\bullet),\quad \underset{i}\Sigma(\beta_i)_\mathfrak p=\xi_p.\end{equation}
Since $\forall k\geq 0$, $\pi_k(A_\bullet)$ is locally complete, the corresponding family of homology classes $\{[\beta_i]\}$ adds to one class $[\beta]\in H_\bullet(\mathcal M(A_\bullet))$. It is clear that $\gamma(\beta)$ is homologous to $\xi$.\end{proof}

\bigskip

We have started with the category $s\mathbb G$, and we were interested in the simplicial localization of $s\mathbb G$ with respect to local weak equivalences. We have constructed two adjunctions:
	\begin{equation}\overline{s\mathbb G}\rightleftarrows s\mathcal G^{op},\quad\overline{s\mathbb G}\rightleftarrows s\mathbb G.\end{equation}
In each case the unit and counit consist of local weak equivalences (in the case of $s\mathbb G$, $\overline{s\mathbb G}$), or weak equivalences (in the case of $s\mathcal G$). Therefore simplicial localization of $s\mathcal G^{op}$, with respect to weak equivalences, is weakly equivalent to simplicial localization of $s\mathbb G$ with respect to local weak equivalences. 

\bigskip

We finish this section with technical results, that were used in some of the propositions above. 

\begin{proposition}\label{InvertingModulo} Let $A\in\mathcal C^\infty\mathcal R$, and let $\mathfrak I\subset A$ be an ideal, s.t. $A/\mathfrak I$ is finitely generated. Let $V\subseteq Sp(A/\mathfrak I)$ be closed, and let 
	\begin{equation}\Lambda:=\{a\in A\text{ s.t. }\forall p\in V\text{ }\mathfrak p(a)\neq 0\}.\footnote{Here $\mathfrak p$ is the composition 
	$A\rightarrow A/\mathfrak I\overset{p}\rightarrow\mathbb R$.}\end{equation}  
Let $A_V$ be obtained by inverting every $a\in\Lambda$. Let $S\subseteq A$ be a set of generators. Then
	\begin{equation}A_V\cong C^\infty(\mathbb R^S)/\mathfrak A+\mathfrak m_{\mathcal V}^g,\end{equation}
where $\mathfrak A\subset C^\infty(\mathbb R^S)$ is the kernel of $C^\infty(\mathbb R^S)\rightarrow A$, and $\mathfrak m_{\mathcal V}^g\subset C^\infty(\mathbb R^S)$ consists of functions, that have $0$ germ at $C^\infty(\mathbb R^S)\rightarrow A\rightarrow A/\mathfrak I\overset{p}\rightarrow\mathbb R$ for any $p\in V$.\end{proposition}
\begin{proof} Since $\mathfrak A\subseteq\mathfrak A+\mathfrak m_\mathcal V^g$, there is a canonical $\phi:A\rightarrow C^\infty(\mathbb R^S)/\mathfrak A+\mathfrak m_\mathcal V^g$. First we prove that $\forall a\in\Lambda$ $\phi(a)$ is invertible. Let $\widetilde{a}\in C^\infty(\mathbb R^S)$ be any pre-image of $a$. Since $A/\mathfrak I$ is finitely generated, we can choose a finite $F\subseteq S$, s.t. $C^\infty(\mathbb R^F)\hookrightarrow C^\infty(\mathbb R^S)\rightarrow A/\mathfrak I$ is surjective, and $\widetilde{a}\in C^\infty(\mathbb R^F)$. Let $\mathcal V_F$ be the image of $V$ in $\mathbb R^F$. Since $\widetilde{a}(p)\neq 0$ $\forall p\in\mathcal V$, using partition of unity on $\mathbb R^F$, we can find $\widetilde{b}\in C^\infty(\mathbb R^F)$, s.t. $\widetilde{a}\widetilde{b}-1=0$ in a neighborhood of $\mathcal V_F$. This implies that $\widetilde{a}\widetilde{b}-1$, considered as an element of $C^\infty(\mathbb R^S)$, has $0$ germ in a neighborhood of every $p\in\mathcal V$, i.e. $\widetilde{a}\widetilde{b}-1\in\mathfrak m_\mathcal V^g$, and hence $\phi(a)$ is invertible in $C^\infty(\mathbb R^S)/\mathfrak A+\mathfrak m_\mathcal V^g$. 

Since $\phi$ is obviously surjective, for any $A\rightarrow B$, that inverts every $a\in\Lambda$, there is at most one factorization through $\phi$. To prove that at least one such factorization exists, we construct $\chi:C^\infty(\mathbb R^S)/\mathfrak A+\mathfrak m_\mathcal V^g\rightarrow A_V$, s.t. the following diagram is commutative. 
	\begin{equation}\xymatrix{& A\ar[ld]_{\psi_0}\ar[rd]^\phi &\\ A_V && C^\infty(\mathbb R^S)/\mathfrak A+\mathfrak m_\mathcal V^g\ar[ll]^{\chi}}\end{equation} 
To construct this $\chi$, it is enough to show that kernel of $C^\infty(\mathbb R^S)\rightarrow A_V$ contains $\mathfrak m_\mathcal V^g$. Let $\widetilde{a}\in\mathfrak m_\mathcal V^g$, and choose a finite $F\subseteq S$, s.t. $\widetilde{a}\in C^\infty(\mathbb R^F)$, and $C^\infty(\mathbb R^F)\rightarrow A/\mathfrak I$ is surjective, let $\mathcal V_F$ be the image of $V$ in $\mathbb R^F$. By assumption, there is an open $\mathcal U_F\supseteq\mathcal V_F$, s.t. $\widetilde{a}=0$ on $\mathcal U_F$. Using paracompactness of $\mathbb R^F$, we can find $\widetilde{b}\in C^\infty(\mathbb R^F)$, s.t. $\widetilde{b}(p)\neq 0$ $\forall p\in\mathcal V_F$ and $supp(\widetilde{b})\subseteq\mathcal U_F$. This means that $\widetilde{a}\widetilde{b}=0$, yet the image of $\widetilde{b}$ in $A_V$ is invertible. Hence $\widetilde{a}\mapsto 0\in A_V$.\end{proof}

\begin{proposition}\label{FullLocalization} Let $A_\bullet$ be a simplicial $C^\infty$-ring, s.t. $\pi_0(A_\bullet)$ is finitely generated and germ determined. Let $\overline{A}_\bullet$ be obtained by universally inverting every $a\in A_0$, s.t. $\mathfrak p(a)\neq 0$, $\forall p\in Sp(\pi_0(A_\bullet))$. Then the natural map $\phi_\bullet:A_\bullet\rightarrow\overline{A}_\bullet$ is a weak equivalence.\end{proposition}
\begin{proof} From Proposition \ref{InvertingModulo} we know that $\overline{A}_0\cong A_0/\mathfrak m$, for some ideal $\mathfrak m\subseteq A_0$. It is easy to see that $\overline{A}_\bullet\cong A_\bullet/\mathfrak mA_\bullet$. 

Clearly $\mathfrak m\mathcal M(A_\bullet)$ is a subcomplex, and hence to prove the proposition it is enough to show that $\mathfrak m\mathcal M(A_\bullet)$ is acyclic. Let $k\geq 0$, and let $\beta=\underset{1\leq i\leq n}\Sigma b_i \beta_i$, where $b_i\in\mathfrak m$ and $\beta_i\in A_k$. Since $\pi_0(A_\bullet)$ is finitely generated, and there are only finitely many of $b_i$'s, one can use partition of unity to find $a\in A_0$, s.t. $a\beta=0$, yet $\mathfrak p(a)\neq 0$ $\forall\mathfrak p\in\mathcal V$. The latter implies that the image of $a$ in $\pi_0(A_\bullet)$ is invertible ($\pi_0(A_\bullet)$ is assumed to be germ determined), and hence, if $\beta$ is a cycle, it is necessarily a trivial one.

Now assume there is $\alpha\in A_{k+1}$, s.t. $d\alpha=\beta$. Using partition of unity again, we can find $b\in\mathfrak m$, s.t. $b b_i=b_i$ for all $1\leq i\leq n$. 
Then obviously $b\alpha\in\mathfrak mA_{k+1}$ and $d(b\alpha)=b\beta=\beta$.\end{proof}


\section{Spivak's construction}\label{ThirdSection}

In this section we show that construction of derived manifolds, done in \cite{Sp10}, can be equivalently performed in the category $s\mathcal L$ of simplicial $C^\infty$-rings $A_\bullet$, s.t. $\pi_0(A_\bullet)$ is finitely generated.

The meeting point of Spivak's construction and the usual homotopy theory of simplicial $C^\infty$-rings is the category $s\mathbb G$ of simplicial $C^\infty$-spaces, defined in Section \ref{SecondSection}. We start with recalling (and somewhat reformulating) the constructions of \cite{Sp10}.

\bigskip

Let $\mathbf{CG}$ be the category of compactly generated Hausdorff spaces. We define a category $\mathbf{RS}$ as follows:\begin{itemize}
\item[1.] objects are pairs $(X,O_{\bullet,X})$, where $X\in\mathbf{CG}$ and $O_{\bullet,X}$ is a functor
	\begin{equation}\xymatrix{Open(X)^{op}\times\mathcal C^\infty\ar[r] & SSet,}\end{equation}
\item[2.] morphisms are pairs $(\phi,\phi^\sharp)$, where $\phi:X\rightarrow Y$ is a continuous map, and $\phi^\sharp$ is a natural transformation
	\begin{equation}\xymatrix{Open(Y)^{op}\times\mathcal C^\infty\ar[rr]^{\phi^{-1}\times Id}\ar[rdd]
	&& Open(X)^{op}\times\mathcal C^\infty\ar[ldd]\\ 
	&\overset{\phi^\sharp}\longrightarrow&\\ 
	& SSet &}\end{equation}
\end{itemize}
For any fixed $X\in\mathbf{CG}$, we will denote by $\mathbf{RS}(X)\subset\mathbf{RS}$ the full subcategory consisting of pre-sheaves on $X$. We equip each $\mathbf{RS}(X)$ with the injective closed model structure, where $O_{\bullet,X}\rightarrow O'_{\bullet,X}$ is a weak equivalence or cofibration, if $O_{\bullet,X}(U,\mathbb R^n)\rightarrow O'_{\bullet,X}(U,\mathbb R^n)$ is respectively a weak equivalence or cofibration of simplicial sets, $\forall U\in Open(X)$, $\forall n\geq 0$.

\smallskip

Since we have not required $O_{\bullet,X}(U,-):\mathcal C^\infty\rightarrow SSet$ to be product preserving, the natural maps 
	\begin{equation}\label{LocProduct}\xymatrix{O_{\bullet,X}(U,\mathbb R^{m+n})\ar[r] & O_{\bullet,X}(U,\mathbb R^m)\times O_{\bullet,X}(U,\mathbb R^n)}\end{equation}
are not required to be isomorphisms, and not even weak equivalences, i.e. $O_{\bullet,X}$ is not necessarily a pre-sheaf of $C^\infty$-rings. Similarly, having a hypercover $\{U_i\}$ of $U$, the natural map
	\begin{equation}\label{LocSheaf}O_{\bullet,X}(U,\mathbb R^k)\rightarrow holim(O_{\bullet,X}(U_i,\mathbb R^k))\end{equation}
does not have to be a weak equivalence, i.e. $O_{\bullet,X}(-,\mathbb R^k)$ is not necessarily a homotopy sheaf of simplicial sets.

Using the fact that $\mathbf{RS}(X)$ is a left proper, cellular, simplicial closed model category, one can perform a left localization of $\mathbf{RS}(X)$ with respect to (\ref{LocProduct}) and (\ref{LocSheaf}) (e.g. \cite{Hi09}). The result is a left proper, simplicial closed model category, that we will denote by $\mathbf{Shv(X)}$. Moreover, any continuous map $\phi:X\rightarrow Y$ induces a Quillen adjunction
	\begin{equation}\phi^*:\mathbf{Shv}(Y)\rightleftarrows\mathbf{Shv}(X):\phi_*.\end{equation}
\underline{Homotopy sheaves of homotopy simplicial $C^\infty$-rings} on $X$ are the fibrant objects in $\mathbf{Shv}(X)$. Every $O_{\bullet,X}\in\mathbf{Shv}(X)$ is cofibrant, and we will denote by $\mathcal O_{\bullet, X}$ a chosen functorial fibrant replacement. For any $U\in Open(X)$, $\mathcal O_{\bullet,X}(U,-)$ is a homotopy $C^\infty$-ring, i.e. (\ref{LocProduct}) is a weak equivalence $\forall m,n\geq 0$, and therefore we have a well defined sheaf of $C^\infty$-rings $\pi_0(\mathcal O_{\bullet,X})$, and a sequence of sheaves of $\pi_0(\mathcal O_{\bullet,X})$-modules $\{\pi_k(\mathcal O_{\bullet,X})\}_{k\geq 1}$. 

\smallskip

\underline{A weak equivalence in $\mathbf{RS}$} is a morphism $(\phi,\phi^\sharp):(X,O_{\bullet,X})\rightarrow(Y,O_{\bullet,Y})$, s.t. $\phi$ is a homeomorphism, and $\phi^\sharp: O_{\bullet,Y}\rightarrow\phi_*(O_{\bullet,X})$ is a weak equivalence in $\mathbf{Shv}(Y)$. Equivalently, we can demand that $\mathcal O_{\bullet,Y}\rightarrow\phi_*(\mathcal O_{\bullet,X})$ is a local weak equivalence, i.e. it induces isomorphisms 
	\begin{equation}\pi_k(\mathcal O_{\bullet,Y})\rightarrow\phi_*(\pi_k(\mathcal O_{\bullet,X})),\quad\forall k\geq 0.\end{equation}
We will denote by $\underline{\mathbf{RS}}$ the simplicial localization of $\mathbf{RS}$ with respect to these weak equivalences. Presence of simplicial closed model structure on each $\mathbf{Shv}(X)$ makes computing $\underline{\mathbf{RS}}$ easier than usual. 

\begin{proposition}\label{CalculusFractions} The category $\mathbf{RS}$, together with the subcategory of weak equivalences, admits a homotopy calculus of fractions.\end{proposition}
\begin{proof} Let $(\phi,\phi^\sharp):(X,O_{\bullet,X})\rightarrow(Y,O_{\bullet,Y})$ be a weak equivalence in $\mathbf{RS}$, we will say that $(\phi,\phi^\sharp)$ is a trivial cofibration or trivial fibration, if correspondingly $\phi^\sharp$ is a trivial fibration or a trivial cofibration. Using closed model structure on $\mathbf{Shv}(Y)$, it is obvious that every weak equivalence in $\mathbf{RS}$ can be written as a composition of a trivial cofibration, followed by a trivial fibration.

Consider diagrams
	\begin{equation}\label{Squares}\xymatrix{(X,O_{\bullet,X})\ar[r]^{(\alpha,\alpha^\sharp)}\ar[d]_{(\phi,\phi^\sharp)} & 
	(Y,O_{\bullet,Y}) & & (Z,O_{\bullet,Z})\ar[d]^{(\psi,\psi^\sharp)}\\ (X',O_{\bullet,X'}) & & (W,O_{\bullet,W})\ar[r]_{(\beta,\beta^\sharp)} & (Z',O_{\bullet,Z'})}\end{equation}
where $(\phi,\phi^\sharp)$ is a trivial cofibration, and $(\psi,\psi^\sharp)$ is a trivial fibration. It is easy to see that 
	\begin{equation}(Y,O_{\bullet,Y}\underset{\alpha_*(O_{\bullet,X})}\prod(\alpha\phi^{-1})_*(O_{\bullet,X'})),\quad(W,O_{\bullet,W}\underset
	{\beta^*(O_{\bullet,Z'})}\coprod(\psi^{-1}\beta)^*(O_{\bullet,Z}))\end{equation}
are respectively colimit and limit of (\ref{Squares}). Since right Quillen functors preserve trivial fibrations, and left Quillen functors preserve trivial cofibrations, it is clear that pushout of $(\phi,\phi^\sharp)$ is a trivial cofibration, and pullback of $(\psi,\psi^\sharp)$ is a trivial fibration. Using the $2$-out-of-$3$ property, we see that, if in addition $(\alpha,\alpha^\sharp)$, $(\beta,\beta^\sharp)$ are weak equivalences, their pushout and pullback are weak equivalences as well. Use \cite{DK80}, proposition 8.2.\end{proof}

\smallskip

The same argument shows that any subcategory of $\mathbf{RS}$, defined by putting conditions on the sheaves of homotopy groups, also admits a homotopy calculus of fractions. We are interested in the following two subcategories:\begin{itemize}
\item let $\mathbf{LRS}\subset\mathbf{RS}$ be the full subcategory, consisting of pairs $(X,O_{\bullet,X})$, s.t. stalks of $\pi_0(\mathcal O_{\bullet,X})$ are local $C^\infty$-rings,
\item let $\mathbf{LRS}_{fgs}\subset\mathbf{LRS}$ be the full subcategory of pairs $(X,O_{\bullet,X})$, s.t. $\pi_0(\mathcal O_{\bullet,X})$ is a soft sheaf of finitely generated $C^\infty$-rings.\end{itemize}
As with $\mathbf{RS}$, we will denote by $\underline{\mathbf{LRS}}$, $\underline{\mathbf{LRS}}_{fgs}$ the simplicial localizations with respect to local weak equivalences. Because $\mathbf{RS}$ admits a homotopy calculus of fractions, we know that $Hom_{\underline{\mathbf RS}}((X,O_{\bullet,X}),(Y,O_{\bullet,Y})$ is weakly equivalent to simplicial set of hammocks of the following form:
	\begin{equation}\label{Hammock}\xymatrix{& (X',O_{\bullet,X'})\ar[ldd]\ar[r]\ar[d] & (Y',O_{\bullet,Y'})\ar[d] &\\
	& (X'',O_{\bullet,X''})\ar[ld]\ar[r]\ar[d] & (Y'',O_{\bullet,Y''})\ar[d] &\\
	(X,O_{\bullet,X}) &\ar@{.}[dd]&\ar@{.}[dd]& (Y,O_{\bullet,Y})\ar[luu]\ar[lu]\ar[ldd]\\
	&&&\\
	& (X^{(n)},O_{\bullet,X^{(n)}})\ar[luu]\ar[r] & (Y^{(n)},O_{\bullet,Y^{(n)}}) &}\end{equation}
where vertical arrows, and arrows going to the left are weak equivalences. The same is true for $\mathbf{LRS}$ and $\mathbf{LRS}_{f g s}$. Since $\mathbf{LRS}$, $\mathbf{LRS}_{f g s}$ are full subcategories of $\mathbf{RS}$, defined by a condition on weak equivalence classes, we immediately have the following result.

\begin{proposition} The inclusions $\underline{\mathbf{LRS}}_{fgs}\subset\underline{\mathbf{LRS}}\subset\underline{\mathbf{RS}}$ induce weak equivalences on the spaces of morphisms.\end{proposition}
So far we have used only presence of a closed model structure on each $\mathbf{Shv}(X)$. Now we will use the fact that these closed model structures are simplicial. For any $(X,O_{\bullet,X})$, $(Y,O_{\bullet,Y})\in\mathbf{RS}$ we have a simplicial set
	\begin{equation}\underset{\phi\in Hom_{\mathbf{CG}}(X,Y)}\coprod\quad\underset{k\geq 0}\coprod
	Hom_{\mathbf{Shv}(Y)}(O_{\bullet,Y}\otimes\Delta[k],\phi_*(\mathcal O_{\bullet,X})),\end{equation}
which we will denote by $\underline{Hom}_{\mathbf{RS}}((X,O_{\bullet,X}),(Y,O_{\bullet,Y}))$.

\begin{proposition} For any $(X,O_{\bullet,X})$, $(Y,O_{\bullet,Y)}\in\mathbf{RS}$ there is a weak equivalence of simplicial sets:
	\begin{equation}\underline{Hom}_{\mathbf{RS}}((X,O_{\bullet,X}),(Y,O_{\bullet,Y}))\simeq Hom_{\underline{\mathbf{RS}}}((X,O_{\bullet,X}),(Y,O_{\bullet,Y})).\end{equation}
\end{proposition}
\begin{proof} Recall from the proof of Proposition \ref{CalculusFractions}, that $(\phi,\phi^\sharp):(X,O_{\bullet,X})\rightarrow(Y,O_{\bullet,Y})$ is a trivial fibration, if $\phi$ is a homeomorphism, and $\phi^\sharp:O_{\bullet,Y}\rightarrow\phi^*(O_{\bullet,X})$ is a trivial cofibration. Since cofibrations in each $\mathbf{Shv}(X)$ are just injective maps, it is easy to see that closing trivial fibrations in $\mathbf{RS}$ with respect to the $2$-out-of-$3$ property, produces all weak equivalences. Hence $\underline{\mathbf{RS}}$ can be computed as simplicial localization of $\mathbf{RS}$ with respect to trivial fibrations. 

Proceeding as in the proof of Proposition \ref{CalculusFractions}, one sees that $\mathbf{RS}$, together with trivial fibrations, admits a calculus of homotopy right fractions (\cite{DK80}, proposition 8.1). Therefore, $Hom_{\underline{\mathbf{RS}}}((X,O_{\bullet,X}),(Y,O_{\bullet,Y})$ is weakly equivalent to simplicial set of hammocks of the following form:
	\begin{equation}\label{SmallHammock}\xymatrix{ & (Y',O_{\bullet,Y'})\ar@{<-}[ldd]\ar[d] &\\
	& (Y'',O_{\bullet,Y''})\ar@{<-}[ld]\ar[d] &\\
	(X,O_{\bullet,X}) & \ar@{.}[dd] & (Y,O_{\bullet,Y})\ar[luu]\ar[lu]\ar[ldd]\\
	&&&\\
	& (Y^{(n)},O_{\bullet,Y^{(n)}})\ar@{<-}[luu] &}\end{equation}
where vertical arrows, and arrows going to the left are trivial fibrations. Moreover, since nerves of equivalent categories are weakly equivalent, we can assume that $Y^{(k)}=Y$ $\forall k\geq 1$. 

It is easy to see that in each such hammock, every path from $X$ to $Y$ has the same underlying continuous map $\phi:X\rightarrow Y$, and hence we have a decomposition of simplicial sets
	\begin{equation}Hom_{\underline{\mathbf{RS}}}((X,O_{\bullet,X}),(Y,O_{\bullet,Y}))=\underset{\phi}\coprod\text{ }
	Hom_{\underline{\mathbf{RS}}}((X,O_{\bullet,X}),(Y,O_{\bullet,Y}))_\phi.\end{equation}
Fix a $\phi$, using functorial fibrant replacement in each $\mathbf{Shv}(X)$, we can assume that all pre-sheaves in (\ref{SmallHammock}) are fibrant. Pushing forward everything to $Y$, we see that (\ref{SmallHammock}) becomes a hammock between $(Y,\phi_*(\mathcal O_{\bullet,X}))$ and $(Y,\mathcal O_{\bullet,Y})$ in $\mathbf{Shv}(Y)$, and this correspondence is bijective, i.e. we have
	\begin{equation}Hom_{\underline{\mathbf{RS}}}((X,\mathcal O_{\bullet,X}),(Y,\mathcal O_{\bullet,Y}))\simeq 
	Hom_{\underline{\mathbf{Shv}(Y)}}(\mathcal O_{\bullet,Y},\phi_*(\mathcal O_{\bullet,X})),\end{equation}
where $\underline{\mathbf{Shv(Y)}}$ is the simplicial localization of $\mathbf{Shv}(Y)$ with respect to trivial cofibrations. Finally, since $\mathcal O_{\bullet,Y}$ is cofibrant, and $\phi_*(\mathcal O_{\bullet,X})$ is fibrant, we have 
\begin{equation}Hom_{\underline{\mathbf{Shv}(Y)}}(\mathcal O_{\bullet,Y},\phi_*(\mathcal O_{\bullet,X}))\simeq\underset{k\geq 0}\coprod
	Hom_{\mathbf{Shv}(Y)}(\mathcal O_{\bullet,Y}\otimes\Delta[k],\phi_*(\mathcal O_{\bullet,X}))\end{equation}
Here we use \cite{DK80b}, proposition 5.2, corollary 4.7.\end{proof}
	
\smallskip

In \cite{Sp10} $\mathbf{LRS}$ is equipped with simplicial structure, which we will denote by $\widehat{\mathbf{LRS}}$. We have seen now that $\widehat{\mathbf{LRS}}$ is weakly equivalent to $\underline{\mathbf{LRS}}$, that we have constructed here. Therefore, all constructions involving homotopy limits (e.g. derived manifolds), that one can perform in $\widehat{\mathbf{LRS}}$, can be equivalently performed in $\underline{\mathbf{LRS}}$. 

Of course, the advantage of Spivak's construction is that it is more manageable, than simplicial localization. Now we show that there is another model, which is more manageable still.

\smallskip

Recall (\cite{Sp10}) that \underline{an affine derived manifold} is a homotopy limit (in $\widehat{\mathbf{LRS}}$) of a diagram
	\begin{equation}\label{AffineDerived}\xymatrix{& \mathbb R^0\ar[d]\\ \mathbb R^m\ar[r] &\mathbb R^n}\end{equation}
Since $\underline{\mathbf{LRS}}\simeq\widehat{\mathbf{LRS}}$, we can use $\underline{\mathbf{LRS}}$ instead. Moreover, if $(X,\mathcal O_{\bullet,X})$ is the homotopy pullback of (\ref{AffineDerived}), then $\pi_0(\mathcal O_{\bullet,X})$ is a soft sheaf of finitely generated $C^\infty$-rings. Therefore affine derived manifolds lie in the full subcategory $\underline{\mathbf{LRS}}_{f g s}\subset\underline{\mathbf{LRS}}$.

Arbitrary derived manifolds in \cite{Sp10} are defined by gluing affine ones. So if we restrict to derived manifolds \underline{of finite type}, i.e. $(X,\mathcal O_{\bullet,X})$ such that $\pi_0(\mathcal O_{\bullet,X})$ is a sheaf of finitely generated $C^\infty$-rings, we are still inside $\underline{\mathbf{LRS}}_{f g s}$.

\smallskip

There is an obvious (full) inclusion 
	\begin{equation}\label{RectInclusion} s\mathbb G\subset\mathbf{LRS}_{fgs}.\end{equation} 
We claim that (\ref{RectInclusion}) induces a weak equivalence of simplicial localizations (with respect to local weak equivalences). Indeed, any functor $\mathcal C^\infty\rightarrow SSet$ can be rectified to a simplicial $C^\infty$-ring, i.e. it can be changed into a product preserving functor. Moreover, this process is functorial, and it preserves homotopy type (\cite{Be06}). 

Let $<s\mathbb G>\subset\mathbf{LRS}$ be the full subcategory, consisting of $s\mathbb G$ and all pairs $(X,O_{\bullet,X})$, s.t. $O_{\bullet,X}$ is weakly product preserving and cofibrant in the {\it projective} closed model structure on $\mathbf{RS}(X)$. Then $s\mathbb G\subseteq<s\mathbb G>$ induces a weak equivalence between simplicial localizations with respect to local weak equivalences (\cite{Lu09}, lemma 5.5.9.9). On the other hand, using cofibrant replacement functor $\mathbf{LRS}\rightarrow\mathbf{LRS}$ (with respect to the projective model structures on $\mathbf{RS}(X)$), we conclude that $<s\mathbb G>\subset\mathbf{LRS}_{fgs}$ also induces a weak equivalences between simplicial localizations. 

\smallskip

Altogether, using results of Section \ref{SecondSection}, we now know that taking homotopy limit of (\ref{AffineDerived}) in $\widehat{\mathbf{LRS}}$ is equivalent to doing the same in $s\mathcal G^{op}$, or equivalently in $s\mathcal L^{op}$, which has the usual (projective) closed model structure. Therefore: 

\begin{theorem} The simplicial category of derived manifolds of finite type, as defined in \cite{Sp10}, is weakly equivalent to the full subcategory of $s\mathcal G^{op}$, consisting of objects, that locally are homotopy limits (taken in $s\mathcal L$) of (\ref{AffineDerived}).\end{theorem}


\end{document}